\documentclass[12pt]{amsart}
\usepackage{amsmath,amssymb,amsbsy,amsfonts,amsthm,latexsym,amsopn,amstext,amsxtra,epic, euscript,amscd,indentfirst}
\usepackage{graphicx}
\begin{document}
\renewcommand{\emptyset}{\varnothing}
\newtheorem{theorem}{Theorem}
\newtheorem{defn}[theorem]{Definition}
\newtheorem{conjecture}[theorem]{Conjecture}
\newtheorem{proposition}[theorem]{Proposition}
\newtheorem{question}[theorem]{Question}
\newtheorem{lemma}[theorem]{Lemma}
\newtheorem{cor}[theorem]{Corollary}
\newtheorem{obs}[theorem]{Observation}
\newtheorem{proc}[theorem]{Procedure}
\newcommand{\comments}[1]{} 
\def\Z{\mathbb Z}
\def\Za{\mathbb Z^\ast}
\def\Fq{{\mathbb F}_q}
\def\R{\mathbb R}
\def\N{\mathbb N}
\def\cH{\overline{\mathcal H}}
\def\cF{\mathcal F}

\title[White's Theorem]{An Exposition of White's characterization of Empty Lattice Tetrahedra}
\author{Mizan R. Khan}
\address{MRK: Department of Mathematical Sciences, Eastern Connecticut State University, Willimantic, CT 06226}
\email{khanm@easternct.edu}
\author{Karen M. Rogers}
\address{KMR: Department of Mathematics and Computer Science, Oxford College at Emory University, Oxford, GA 30054}
\email{karen.m.rogers@emory.edu}

\date{}

\begin{abstract}
We give an exposition of White's characterization of empty lattice tetrahedra. In particular, we describe the second author's proof of White's theorem that appeared in her doctoral thesis~\cite{Rog}.
\end{abstract}

\maketitle

\section{Introduction}

The motivating example is the \emph{lattice} tetrahedron with vertices (0,0,0), (1,0,0), (0,1,0) and $(1,1,c)$ with $c$ being an arbitrary positive integer. We denote this tetrahedron as $T_{1,1,c}$. Regardless of the size of $c$ (and consequently the volume of $T_{1,1,c}$), $T_{1,1,c}$ does not contain any lattice points other than its vertices.  This is in surprising contrast to the situation in $\R^2$ where a \emph{lattice} triangle does not contain any lattice points, other than its vertices, if and only if it has area 1/2. (To see this we invoke Pick's theorem.)

Reeve~\cite{Ree} posed the problem of characterizing such tetrahedra. Some years later White~\cite{Wh} solved this problem. Over the years different authors have given proofs of White's theorem (see~\cite{B-vH,M-S,Rez1,Rez2,Sca}). The second author gave a proof of White's theorem in her doctoral dissertation~\cite{Rog}. In this article we give a detailed exposition of this proof. 

Before stating the relevant theorems we establish some notation and definitions. Let $a,b,c \in \Z$ with $0\leq a,b <c$. We will use $d$ to denote the integer
$$ d = (1-a-b) \mod c, 0 \leq d < c.$$ 
Furthermore, $T_{a,b,c}$ will denote the lattice tetrahedron with vertices $(0,0,0)$, $(1,0,0)$, $(0,1,0)$ and 
$(a,b,c)$. 

\begin{defn}
Following Reznick~\cite{Rez2}, we call a lattice polyhedron that does not contain any lattice points other than its vertices an \emph{empty lattice polyhedron}. Such a polyhedron belongs to a larger set of lattice polyhedra that do not contain any lattice points on their boundary other than the vertices. We call such polyhedra \emph{clean lattice polyhedra}.
\end{defn}

We insert a warning about the the terminology, particularly in the case of tetrahedra. Other names in the literature for empty tetrahedra are \emph{fundamental}, 
\emph{primitive}, \emph{Reeve}. 

\begin{defn} 
An \emph{affine unimodular map} is an affine map 
$$L: \R^3 \rightarrow \R^3 \textrm{ of the form } L(\vec{x}) = M\vec{x} + \vec{u},$$
where $M \in GL_3(\Z)$, $\det(M) =\pm 1$, and $\vec{u}\in \Z^3$.
\end{defn}

We now state the two theorems that we will prove.

\begin{theorem}\label{1st theorem}
Let $T$ be an empty lattice tetrahedron. Then there is an affine unimodular map $L$ such that $L(T)= T_{a,b,c}$, with $0\leq a,b <c$ and 
$\gcd(a,c)=\gcd(b,c)=\gcd(d,c)=1$.
\end{theorem}

\begin{theorem}[White]\label{2nd theorem} The lattice tetrahedron $T_{a,b,c}$ is empty if and only if $\gcd(a,c)=\gcd(b,c)=\gcd(d,c)=1$ and at least one of the following hold:
$$a=1,b=1,c=1,d=1.$$
\end{theorem}

We now state definitions and background results that will be used to prove the two theorems.

\begin{defn} 
A set of lattice points $\{ \vec{v}_1, \ldots, \vec{v}_k \}$ in $\Z^n$ is said to be \emph{primitive} if it is a basis for the sublattice 
$\Z^n \cap ( \R\vec{v}_1 \oplus \dots \oplus \R \vec{v}_k)$. Geometrically, this means that $\{ \vec{v}_1, \ldots, \vec{v}_k \}$ is primitive if and only if the 
parallelepiped spanned by $\vec{v}_1, \ldots, \vec{v}_k $ is empty.
\end{defn}

The following is a list of standard results we will use. The proofs can be found in~\cite[Lectures V, VIII]{Sie}. However, we have rephrased some of the statements. Consequently the reader who consults~\cite{Sie} may need to read the relevant material carefully.

\begin{theorem}\label{background1} Every lattice has an integral basis.\end{theorem}

\begin{theorem}\label{background2} The property of being a lattice basis is preserved under the action of any unimodular transformation, that is, if $\vec{v}_1,\vec{v}_2,\ldots, \vec{v}_n$
is a basis for $\Z^n$ and $T:\R^n \rightarrow \R^n$ is an unimodular transformation, then $T\left(\vec{v}_1\right),T\left(\vec{v}_2\right),\ldots, T\left(\vec{v}_n\right)$ is also
a basis of $\Z^n$. Furthermore, given two lattice bases there is an unimodular transformation that maps one basis into the other.
\end{theorem}

\begin{theorem} \label{background3} Let $\{\vec{v}_1, \ldots, \vec{v}_n\}$ be a linearly independent set of elements of $\Z^n$, and let $H=\Z\vec{v}_1\oplus \ldots \oplus \Z\vec{v}_n$. 
Then the order of the quotient group $\Z^n/H$  equals
$$ \#(\Z^n/H) = |\det(\vec{v}_1, \ldots, \vec{v}_n)|.$$
\end{theorem}

\begin{theorem}\label{background4} Let $\{\vec{v}_1, \ldots, \vec{v}_r\}$ be a primitive set of $\Z^n$. Then $\{\vec{v}_1, \ldots, \vec{v}_r\}$ can be extended to a basis of $\Z^n$.
\end{theorem}

We mention an interesting fact that emerges in the course of proving White's theorem. From Theorem~\ref{background3} it follows that if $T_{a,b,c}$ is empty, then the parallelepiped spanned by $(1,0,0),$ $(0,1,0),$ $(a,b,c)$ contains $(c-1)$ 
lattice points in its interior. In the course of proving Theorem~\ref{2nd theorem} we will find that \emph{all of these points are coplanar}! More precisely, we have the following.

\begin{cor}\label{coplanar}
Let $P_{a,b,c}$ denote the parallelepiped spanned by $(1,0,0),$ $(0,1,0)$ and $(a,b,c)$. If $T_{a,b,c}$ is empty, then $P_{a,b,c}$ contains $(c-1)$ lattice points in its interior. If $a=1$, then all of these lattice points lie on the plane $x=1$; if $b=1$, then all of these lattice points lie on the plane $y=1$; if $d=1$, then all of these lattice points lie on the plane $x+y-z=1$.
\end{cor}

\noindent {\bf Warning:} The co-planarity of these lattice points was mentioned in an article of the first author~\cite[Theorem 3.2]{Kh}. Unfortunately the description of the planes in~\cite{Kh} is completely incorrect! The author should have done his homework and not just relied on his faulty visualization skills!!

\noindent {\bf Acknowledgements.} This article is a variant of an earlier manuscript drafted and submitted in the early 1990's. It was written in collaboration with Therese Hart who did the initial calculations that led us to discover White's result. We withdrew the article after discovering that White had anticipated our main discovery three decades previously. We would like to 
express our gratitude to Mel Nathanson who encouraged us to write this expository piece.

\section{Proofs}

We begin with some notation. Let $\vec{u}= (u_1,u_2,u_3) \in \Z^3$. We will denote the integer $\gcd(u_1,u_2,u_3)$ by $\gcd(\vec{u})$. Occasionally we will use $e_1,e_2$ and $e_3$ to denote the vectors $(1,0,0)$, $(0,1,0)$, and $(0,0,1)$. 

\begin{proposition}\label{useful prop}
Let $\vec{u}, \vec{v}$ be two linearly independent elements in $\Z^3$. The following statements are equivalent.
\begin{enumerate}

\item P, the parallelogram spanned by $\vec{u}$ and $\vec{v}$ is an empty parallelogram.

\item T, The triangle spanned by $\vec{u}$ and $\vec{v}$ is an empty triangle.

\item $\gcd(\vec{u} \times \vec{v})=1$.
\end{enumerate}
\end{proposition}

\begin{proof}
Clearly (1) $\Rightarrow$ (2). We prove the contrapositive to demonstrate that (2) $\Rightarrow$ (1). We assume that $P$ contains a lattice point $\vec{x}$ that is not a vertex of $P$.  
Then either $\vec{x}$ or $(\vec{u}+\vec{v}-\vec{x})$ lies in $T$. Since neither lattice point can be a vertex of $T$, we conclude that $T$ is not an empty triangle.

We now turn to proving that (1) and (3) are equivalent.

(3) $\Rightarrow$ (1): Since $\gcd(\vec{u}\times\vec{v}) =1$, there exists, by the Extended Euclidean algorithm, $\vec{w}\in \Z^3$ 
such that $(\vec{u}\times\vec{v})\cdot \vec{w} =1 = \det(\vec{u},\vec{v},\vec{w})$. 
By Theorem~\ref{background3}, $\vec{u},\vec{v},\vec{w}$ is a basis of $\Z^3$,
and consequently they span an empty parallelepiped. We conclude that $P$ is an empty parallelogram. 

(1) $\Rightarrow$ (3): Since $P$ is an empty parallelogram, $\{\vec{u}, \vec{v}\}$ is a primitive set of $\Z^3$, and consequently by Theorem~\ref{background4} there is a lattice point 
$\vec{w}$ such that $\vec{u},\vec{v},\vec{w}$ is a basis of $\Z^3$. Consequently $|\det(\vec{u},\vec{v},\vec{w})|=1$. Since 
$\det(\vec{u},\vec{v},\vec{w}) = (\vec{u}\times\vec{v})\cdot \vec{w}$, we conclude that $\gcd(\vec{u}\times \vec{v}) =1$.
\end{proof}

\begin{cor}\label{clean-tet}
The tetrahedron $T_{a,b,c}$ is clean if and only if $\gcd(a,c)=\gcd(b,c)=\gcd(d,c)=1$.
\end{cor}

\begin{proof} Let $\triangle_1,\triangle_2,\triangle_3,\triangle_4$ denote the faces of $T_{a,b,c}$ where $\triangle_1$ is the triangle spanned by $e_1$ and $e_2$; $\triangle_2$ is the triangle spanned by $e_1$ and $(a,b,c)$; $\triangle_3$ is the triangle spanned by $e_2$ and $(a,b,c)$; and $\triangle_4$ is the triangle spanned by $(e_2-e_1)$ and $((a,b,c)-e_1)$. $T_{a,b,c}$ is a clean tetrahedron if and only if $\triangle_1,\triangle_2,\triangle_3$ and $\triangle_4$ are all empty lattice triangles. Clearly $\triangle_1$ is an empty triangle. By 
Proposition~\ref{useful prop} the triangles $\triangle_2,\triangle_3,\triangle_4$ are empty if and only if 
$$ \gcd(e_1\times (a,b,c))=\gcd(e_2 \times (a,b,c)) = \gcd((e_2-e_1)\times ((a,b,c)-e_1))=1,$$
that is, $\gcd(b,c)=\gcd(a,c)=\gcd(d,c)=1.$
\end{proof}
 
\begin{proof}[Proof of Theorem~\ref{1st theorem}]
Let $T$ be an empty lattice tetrahedron in $\R^3$. Without loss of generality we may assume that the origin is one of the vertices and the other 3 vertices are 
$\vec{u}, \vec{v}$ and $\vec{w}$. Since the triangle spanned by $\vec{u}$ and $\vec{v}$ is empty, by Proposition ~\ref{useful prop}, the same holds for the parallelogram spanned by
$\vec{u}$ and $\vec{v}$. Therefore
$ \{\vec{u}, \vec{v} \} $ is a primitive set of $\Z^3$, and by Theorem~\ref{background4} can be extended to a basis of $\Z^3$, $\vec{u},\vec{v}, \vec{x}$. Now by Theorem~\ref{background2} we have a unimodular transformation $L_1$ 
such that $L_1(\vec{u})=e_1$, $L_2(\vec{v})=e_2$, and $L_3(\vec{x})=e_3$. Under this transformation we see that the tetrahedron $T$ is equivalent to the tetrahedron $T_1$ with vertices 
$0, e_1, e_2$ and $(A,B,c)$ where $A,B,c \in \Z$ and $\textrm{vol}(T)=|c/6|$. If $c<0$ we can compose $L_1$ with the unimodular transformation 
$$L_2((x,y,z))=(x,y,-z).$$
Consequently we can assume that $c >0$. We now use the division algorithm to express 
$$A = q_1c + a \textrm{ and } B = q_2c +b, 0\leq a,b < c.$$
By acting on $T_1$ by the unimodular transformation $$L_3((x,y,z))=(x-q_1z,y-q_2z,z)$$
we get that $T$ is equivalent to the tetrahedron $T_2$ with vertices $0,e_1,e_2$ and $(a,b,c)$. Since $T_2$ is a clean tetrahedron we invoke Corollary~\ref{clean-tet} 
 to conclude that $\gcd(a,c)=\gcd(b,c)=\gcd(d,c)=1.$
\end{proof}

We now turn to the proof of White's theorem. Our proof is arranged in four parts. These are as follows:
\begin{enumerate}
\item[Part 1:] We prove that the tetrahedron $T_{a,b,c}$ is empty if and only if a system of equations involving $a,b,d$ hold.
\item[Part 2:] This system of equations give an immediate proof of the $(\Leftarrow)$ direction of White's theorem.
\item[Part 3:] The proof of the $(\Rightarrow)$ direction of White's theorem is considerably more involved. We first develop a slight modification of the system of equations. This then leads 
us to define a finite set of arithmetic functions $f_n$. We then state and prove certain properties of these functions. 
\item[Part 4:] We use the properties of $f_n$ to complete the proof.
\end{enumerate}

We will invoke the following identity in several places
\begin{lemma}
Let $x\in \R$. If $x\not\in \Z$, then 
\begin{equation}\label{eq:keyid}
\langle -x \rangle  =  1-\langle x \rangle.
\end{equation}
We will typically invoke this identity in the following form:
\begin{equation}\label{eq:keyidfrac}
\left\langle\frac{kl}{c}\right\rangle + \left\langle\frac{k(c-l)}{c}\right\rangle = 1
\end{equation}
for $0< l < c, \gcd(l,c)=1, \textrm{ and } k=1,\ldots, c-1.$
\end{lemma}

\begin{proposition}
Let $c\in \Z$ with $c>1$ and let $T_{a,b,c}$ be a clean lattice tetrahedron. Then, $T_{a,b,c}$ is empty if and only if   
\begin{equation}\label{eq:sys-eq1}
\left \langle \frac{ka}{c} \right \rangle + \left \langle \frac{kb}{c} \right \rangle +\left \langle \frac{kd}{c} \right \rangle - \frac{k}{c} = 1
\end{equation}
holds for $k=1, \ldots, c-1$.
\end{proposition}

\begin{proof}[Proof of Part 1]
Let $P$ denote the parallelepiped spanned by $e_1, e_2$ and $(a,b,c)$. Since $\textrm{volume}(P)=c$ and the faces of $P$ are empty lattice parallelograms, we infer that $P$ contains $(c-1)$ lattice points in its interior. These lattice points are 
\begin{equation}\label{eq:sys-eq-lat-pt}
\left\langle\frac{k(c-a)}{c}\right\rangle(1,0,0) + \left\langle\frac{k(c-b)}{c}\right\rangle(0,1,0) + \frac{k}{c} (a,b,c)
\end{equation}
with $k=1,\ldots,c-1$. 

$T_{a,b,c}$ is empty if and only if 
$$ 1 < \left\langle\frac{k(c-a)}{c}\right\rangle + \left\langle\frac{k(c-b)}{c}\right\rangle + \frac{k}{c} < 2,$$
for $k=1,\ldots, c-1.$ 
Some algebraic manipulation in conjunction with identity~\eqref{eq:keyid} gives the the system of inequalities 
$$ 0 < \left\langle\frac{ka}{c}\right\rangle + \left\langle\frac{kb}{c}\right\rangle - \frac{k}{c} < 1,$$ 
for $k=1,\ldots, c-1.$ We now observe that 
\begin{equation} 
\left\langle\frac{ka}{c}\right\rangle + \left\langle\frac{kb}{c}\right\rangle - \frac{k}{c} \equiv \left\langle\frac{k(a +b-1)}{c}\right\rangle \pmod{\Z},
\end{equation}
for $k=1,\ldots, c-1.$ Since both sides of the congruence are between 0 and 1, we conclude that we have a system of \emph{equalities}
$$\left\langle\frac{ka}{c}\right\rangle + \left\langle\frac{kb}{c}\right\rangle - \frac{k}{c} = \left\langle\frac{k(a +b-1)}{c}\right\rangle,$$
for $k=1,\ldots, c-1.$ After a little more algebraic manipulation we conclude that $T_{a,b,c}$ is empty if and only if 
$$ \left \langle \frac{ka}{c} \right \rangle + \left \langle \frac{kb}{c} \right \rangle +\left \langle \frac{kd}{c} \right \rangle - \frac{k}{c} = 1$$
for $k=1,\ldots,c-1$.
\end{proof}

We can now easily prove $(\Leftarrow)$ direction of White's theorem. The system of equations~\eqref{eq:sys-eq1} in conjunction with the system of 
identities~\eqref{eq:keyidfrac} allow us to conclude that the following tetrahedra are empty.

\begin{cor}
Let $\gcd(a,c)=1$. Then the tetrahedra $T_{1,a,c}$ and $T_{a,c-a,c}$ are empty.
\end{cor}

To prove the $(\Rightarrow)$ direction of White's theorem we will work with a modification of~\eqref{eq:sys-eq1}. Define 
a set of arithmetic functions $f_n$ for $n\in \Z^+, n < c$ and $\gcd(n,c)=1$,  
$$f_n: \{1, \ldots, c-2 \} \rightarrow \{0,1\}$$ 
via
\begin{equation}
f_n(k) = \left \langle \frac{kn}{c} \right \rangle - \left \langle \frac{(k+1)n}{c} \right \rangle + \frac{n}{c} =\left[\frac{(k+1)n}{c}\right] - \left[\frac{kn}{c}\right].
\end{equation}

From~\eqref{eq:sys-eq1} we obtain the system of equations
\begin{equation}\label{eq:sys-eq2}
f_a(k)+f_b(k)+f_d(k) +\frac{1}{c} = \frac{a+b+d}{c}, 
\end{equation}
for $k=1,\ldots, c-2$. We now look at the case of $k=1$ in~\eqref{eq:sys-eq1} which shows that 
$$\frac{a+b+d}{c}= 1+\frac{1}{c}.$$ 
Thus we can rewrite~\eqref{eq:sys-eq2} as the system of equations 
\begin{equation}\label{eq:sys-eq3}
f_a(k)+f_b(k)+f_d(k) = 1, 
\end{equation}
for $k=1,\ldots,c-2.$ We will work with this system~\eqref{eq:sys-eq3} in conjunction with the properties of $f_n$ to 
arrive at a proof of White's theorem.

\begin{proposition}
The function $f_n$ has the following properties.
\begin{enumerate}
\item[(i)] $f_1^{-1}(\{1\}) =\emptyset.$
\item[(ii)] For $n >1$, $$f_n^{-1}(\{1\}) = \left\{ \,\left[kc/n \right] : k=1,\ldots, n-1 \right\}. $$
\item[(iii)] $f_{c-n} =1-f_n$.
\end{enumerate}
\end{proposition}

\begin{proof}
For $k=1,\ldots,c-2$, 
$$f_1(k)= \left[\frac{k+1}{c}\right] - \left[\frac{k}{c}\right] = 0-0 = 0,$$
which proves (i).

We now prove statement (ii). If $l \in f_n^{-1}(\{1\})$ then there exists $k \in \Z^{+}$ such that 
$$ \frac{ln}{c} < k < \frac{(l+1)n}{c}.$$
It follows that $l = [kc/n].$ Conversely, if $l = [kc/n]$ for some integer $k$, with $1 \leq k \leq n-1$, then 
we have that 
$$l < \frac{kc}{n} < l+1.$$
We now obtain that 
$$ \frac{ln}{c} < k < \frac{(l+1)n}{c}$$
and consequently $l \in f_n^{-1}(\{1\})$. 

Statement (iii) is a consequence of identity~\eqref{eq:keyidfrac}.
\begin{eqnarray*}
f_{c-n}(k) & = & \left \langle \frac{k(c-n)}{c} \right \rangle - \left \langle \frac{(k+1)(c-n)}{c} \right \rangle + \frac{c-n}{c} \\ 
& = & 1- \left \langle \frac{kn}{c} \right \rangle -1 + \left \langle \frac{(k+1)n}{c} \right \rangle + 1-\frac{n}{c} \\
& = & 1-f_n(k).
\end{eqnarray*}
\end{proof}

We now complete the proof of White's theorem.

\begin{proof}[Proof of Part 3]
Let $T_{a,b,c}$ be an empty tetrahedron with $c \ge 2$. We want to prove that either $a=1$ or $b=1$ or $d=1$. We will argue by contradiction. So we assume that $a,b,d\ge 2$. Consequently none of the sets 
$f^{-1}_a(\{1\})$,  $f^{-1}_b(\{1\})$, $f^{-1}_d(\{1\})$ are empty. Since 
$$ f_a + f_b + f_d =1,$$ 
can infer that $a,b$ and $d$ are distinct integers and the sets $$f_a^{-1}(\{1\}),  f_b^{-1}(\{1\}),f_d^{-1}(\{1\})$$
are pairwise disjoint. ({\bf Spoiler alert:} Our argument hinges crucially on the fact that $f_b^{-1}(\{1\})\cap f_d^{-1}(\{1\}) = \emptyset.$) 
Without loss of generality we can assume that $a > b > d.$ It follows that $1 \in f_a^{-1}(\{1\})$, and consequently $1\not\in \left( f_b^{-1}(\{1\}) \cup f_d^{-1}(\{1\}) \right)$. 
We now have that $$f_b + f_d = f_{c-a}$$ and consequently $$\left( f_b^{-1}(\{1\}) \cup f_d^{-1}(\{1\}) \right) = f_{c-a}^{-1}(\{1\}),$$ that is, 
$$ \left\{ \,\left[kc/b \right] : k=1,\ldots, b-1 \right\} \cup \left\{ \,\left[kc/d \right] : k=1,\ldots, d-1 \right\} $$
$$=\left\{ \,\left[kc/(c-a) \right] : k=1,\ldots, (c-a-1) \right\}.$$
We now compare the smallest and largest elements in each of the 3 sets. Since $b > d \ge 2$ and $1\not\in f_{c-a}^{-1}(\{1\})$, we have that 
$$2 \le \left[\frac{c}{c-a} \right] = \left[ \frac{c}{b} \right] < \left[ \frac{c}{d} \right] \le  \left[ \frac{(d-1)c}{d} \right] < 
\left[ \frac{(b-1)c}{b} \right] = \left[ \frac{(c-a-1)c}{c-a} \right].$$
We remark that the strict inequalities occur since 
$$f_b^{-1}(\{1\})\cap f_d^{-1}(\{1\}) = \emptyset.$$

Let $s$ be the positive integer such that 
$$\left[ \frac{c}{d} \right] = \left[ \frac{sc}{c-a} \right].$$
We now obtain that 
$$\left[ \frac{(s-1)c}{c-a} \right] = \left[ \frac{(s-1)c}{b} \right] \textrm{ and } \left[ \frac{(s+1)c}{c-a} \right] \le \left[ \frac{sc}{b} \right].$$
Combining these two observations we get 
$$\left[ \frac{(s+1)c}{c-a} \right] - \left[ \frac{(s-1)c}{c-a} \right] \le \left[ \frac{sc}{b} \right] - \left[ \frac{(s-1)c}{b} \right],$$
which implies the inequality 
$$ 2\left[ \frac{c}{c-a} \right] \le \left[ \frac{c}{b} \right] +1.$$
This leads to the contradiction that 
$$ \left[ \frac{c}{c-a} \right] \le 1,$$
and consequently our assumption that $a,b,d \ge 2$ is false.
\end{proof}

\begin{proof}[Proof of Corollary~\ref{coplanar}] Let $T_{a,b,c}$ be empty, with $c>1$. By Theorem~\ref{2nd theorem} we have that either $a=1$ or $b=1$ or $b=c-a$. 
If $a=1$, then by replacing $a$ by 1 in~\eqref{eq:sys-eq-lat-pt} we see that the $x$ co-ordinate of each lattice point inside $P_{1,b,c}$ equals 1. The same argument works if $b=1$. The 
only case that needs a little more work is if $b=c-a$. In this case~\eqref{eq:sys-eq-lat-pt} becomes
\begin{equation}
\left\langle\frac{k(c-a)}{c}\right\rangle(1,0,0) + \left\langle\frac{ka}{c}\right\rangle(0,1,0) + \frac{k}{c} (a,c-a,c).
\end{equation}
If we now add the $x$ and $y$ co-ordinates and subtract the $z$ co-ordinate we get 
$$\left\langle\frac{k(c-a)}{c}\right\rangle + \left\langle\frac{ka}{c}\right\rangle + \frac{ka}{c} +\frac{k(c-a)}{c}-k = \left\langle\frac{k(c-a)}{c}\right\rangle + \left\langle\frac{ka}{c}\right\rangle.$$
We now invoke the identities~\eqref{eq:keyidfrac} to conclude that the RHS equals 1.
\end{proof}

\end{document}